\documentclass[10pt]{amsart}
\usepackage{amsmath,amssymb,amsthm,graphicx,mathrsfs,url}
\usepackage{bbm}
\usepackage[usenames,dvipsnames]{color}
\definecolor{darkred2}{rgb}{0.7,0.1,0.1}
\definecolor{darkred}{rgb}{0.4,0.1,0.1}
\usepackage{rotating}
\usepackage{float}
\usepackage{epic}
\usepackage{tikz}
\usetikzlibrary{intersections,positioning}
\usepackage{caption}

\def\softness{0.4}
\definecolor{softred}{rgb}{1,\softness,\softness}
\definecolor{softgreen}{rgb}{\softness,1,\softness}
\definecolor{softblue}{rgb}{\softness,\softness,1}

\definecolor{softrg}{rgb}{1,1,\softness}
\definecolor{softrb}{rgb}{1,\softness,1}
\definecolor{softgb}{rgb}{\softness,1,1}

\usepackage[latin1]{inputenc}

\usepackage[T1]{fontenc}

\usepackage{graphics}
\numberwithin{figure}{section}
\numberwithin{equation}{section}
\theoremstyle{plain}
\newtheorem*{thm*}{Theorem}
\newtheorem{thm}{Theorem}[section]

\newtheorem{lem}[thm]{Lemma}

\theoremstyle{remark}

\theoremstyle{plain}

\newcommand{\rmd}{\mathrm{d}}
\newcommand{\rmi}{\mathrm{i}}

\newcommand{\supp}{\mathrm{supp}\,}

\newcommand{\drm}{{\mathrm d}}

\newcommand{\be}{\begin{equation}}
\newcommand{\ee}{\end{equation}}
\newcommand{\beu}{\begin{equation*}}
\newcommand{\eeu}{\end{equation*}}
\newcommand{\besu}{\begin{equation*}
\begin{aligned}}
\newcommand{\eesu}{\end{aligned}
\end{equation*}}
\newcommand{\bes}{\begin{equation}
\begin{aligned}}
\newcommand{\ees}{\end{aligned}
\end{equation}}

\newcommand\fra{\mathfrak a}

\newcommand\eps{\varepsilon}

\newcommand\ov{\overline}

\newcommand\sess{\sigma_{\rm ess}}

\newcommand\RE{\text{\rm Re}}

\newcommand\sign{{\rm sign\,}}

\newcommand\void[1]{}

\def\eps{\varepsilon}
\def\sess{\sigma_{\rm ess}}

      \def\dC{{\mathbb C}}

   \def\dN{{\mathbb N}}   
      \def\dR{{\mathbb R}}

      \def\cC{{\mathcal C}}

      \def\cO{{\mathcal O}}

\newcommand{\dom}{\mathrm{dom}\,}

\title[$\delta$-interactions on conical surfaces]{Schr\"odinger operators with $\delta$-interactions supported on conical surfaces}

\author{Jussi Behrndt \and Pavel Exner \and Vladimir Lotoreichik}

\begin{document}

\maketitle
\begin{abstract}
We investigate the spectral properties of self-adjoint Schr\"odinger operators 
with attractive $\delta$-interactions of constant strength 
$\alpha > 0$ supported on conical surfaces in $\dR^3$. It is shown that the essential spectrum is given by $[-\alpha^2/4,+\infty)$ and 
that the discrete spectrum is infinite and accumulates to 
$-\alpha^2/4$. 
Furthermore, an  asymptotic estimate of these eigenvalues is
obtained.
\end{abstract}

\section{Introduction}

The purpose of this paper is to analyse the spectrum of the three-dimensional Schr\"odinger operator $-\Delta_{\alpha,\cC_\theta}$ 
with an attractive $\delta$-interaction of constant strength $\alpha >0$ 
supported on the conical surface 
\[
\cC_\theta :=  
\Big\{(x,y,z) \in\dR^3\colon 
z := \cot(\theta)\sqrt{x^2+y^2}\Big\},\qquad  \theta\in(0,\pi/2).
\]
The Schr\"odinger operator $-\Delta_{\alpha,\cC_\theta}$  is defined via the first representation theorem 
\cite[Theorem VI.2.1]{Kato} as the unique self-adjoint operator in $L^2(\dR^3)$ which
is associated to the closed, densely defined, semibounded quadratic form
\begin{equation}
\label{fra}
\fra_{\alpha,\cC_\theta}[\psi] = 
\|\nabla \psi\|^2_{L^2(\dR^3;\dC^3)} - 
\alpha \int_{\cC_\theta}\,\vert\psi\vert^2 \,\drm\sigma   
\qquad \dom\fra_{\alpha,\cC_\theta}  = H^1(\dR^3);
\end{equation}
cf. \cite{BEL13,BEKS94}.
In a short form the main result of this note is the following theorem.

\begin{thm*}
For any $\theta\in (0,\pi/2)$ and $\alpha>0$ the essential spectrum of the operator $-\Delta_{\alpha,\cC_\theta}$ is
$[-\alpha^2/4, +\infty)$, the discrete spectrum is infinite and accumulates to $-\alpha^2/4$.
\end{thm*}

In addition, we obtain an  asymptotic estimate of the eigenvalues of  $-\Delta_{\alpha,\cC_\theta}$ lying below $-\alpha^2/4$,
and the results also extend to local deformations of the conical surface $\cC_\theta$,
see Theorem~\ref{thm:main} and Theorem~\ref{locdeform} for details.
The proof of our main result is based on standard techniques in spectral theory of self-adjoint operators: we construct 
singular sequences and use Neumann bracketing in the spirit of \cite{EN03} to show the assertion on the essential spectrum;
for the infiniteness of the discrete spectrum we employ variational principles. 
The same approach was applied in \cite{S70} in the context
of Schr\"odinger operators with slowly decaying
negative regular potentials,
see also  \cite[\S XIII.3]{RS-IV}.
Similar arguments were also used in \cite{DEK01, ET10} for the closely related question of infiniteness
of the discrete spectrum for the Dirichlet Laplacian in a conical layer, 
see also  \cite{CEK04, J13, KV08, LL07, LR12} 
for further progress in this problem. We also point out  \cite{BEW09, DR13, EK02} for 
related spectral problems for Schr\"{o}dinger operators with $\delta$-potentials.

\section{Essential spectrum of $-\Delta_{\alpha,\cC_\theta}$}

In this section we show that the
essential spectrum of the operator
$-\Delta_{\alpha,\cC_\theta}$ is given by
$[-\alpha^2/4,+\infty)$. The proof of the inclusion
$\sess(-\Delta_{\alpha,\cC_\theta})\supseteq[-\alpha^2/4,+\infty)$  makes use of 
singular sequences and for the other inclusion a specially chosen Neumann bracketing
is used. A similar type of argument was also used in \cite{BEL13, EN03} for $\delta$ and $\delta'$-interactions
on broken lines in the two-dimensional setting. For completeness we mention that the theorem (and its proof) below is also valid for $\theta=\pi/2$, in which case
the conical surface is a half-plane, and the result is well-known.
\begin{thm}\label{mainthm1}
Let $-\Delta_{\alpha,\cC_\theta}$ be the self-adjoint operator in $L^2(\dR^3)$ associated to the form \eqref{fra}
and let  $\alpha>0$ and $\theta \in (0,\pi/2 )$. 
Then
\[
\sess(-\Delta_{\alpha,\cC_\theta}) = [-\alpha^2/4,+\infty).
\]
\end{thm}

\begin{proof}
\emph{Step 1.} We verify the inclusion $\sess(-\Delta_{\alpha,\cC}) \supseteq [-\alpha^2/4,+\infty)$ 
by constructing singular sequences for the operator 
$-\Delta_{\alpha,\cC_\theta}$ 
for every point of the interval $[-\alpha^2/4,+\infty)$. 
Let us start by fixing a function 
$\chi_1\in C^\infty_0(1,2)$ 
such that 
\begin{equation}
\label{chi1}
\|\chi_1\|_{L^2(1,2)} = 1,
\end{equation}
and a function $\chi_2\in C^\infty_0(-\eps,\eps)$ 
with some fixed $\eps\in(0,\tan\theta)$, 
which satisfies 
\begin{equation}
\label{chi2}
0 \le \chi_2 \le 1\quad \text{and}\quad 
\chi_2(t) = 1~ \text{for} ~|t| < \eps/2.
\end{equation}
Define for all $p\in\dR$ and $n\in\dN$ the functions $\omega_{n,p}\colon\dR^2_+\rightarrow\dC$ as
\begin{equation*}
\omega_{n,p}(s,t) :=
\tfrac{1}{\sqrt{n}}
\Big(\chi_1(\tfrac{s}{n})\exp({\rm i} ps)\Big)
\Big(\chi_2(\tfrac{t}{n})\exp(-\tfrac{\alpha}{2}|t|)\Big) 
\in C(\dR^2_+)
\end{equation*}
in the coordinate system $(s,t)$ in Figure~\ref{fig}. Here $\dR^2_+$ denotes open right half-plane
$\{(r,z)\in\dR^2:r>0\}$.

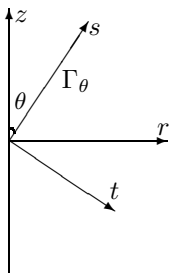
\begin{figure}[H]
\begin{center}
\begin{picture}(150,100)
\put(10,0){\vector(0,1){100}}
\put(10,50){\vector(1,0){60}}
\put(10,50){\vector(2,3){30}}
\put(10,50){\vector(3,-2){40}}
\put(30,70){$\Gamma_\theta$}
\put(40,90){$s$}
\put(12,95){$z$}
\put(66,52){$r$}
\put(48,28){$t$}
\put(12,62){{\small $\theta$}}
\qbezier(10,54)(11,56)(12,53)
\end{picture}
\end{center}
\caption{The right half-space $\dR^2_+$ with the coordinate system $(r,z)$.
The ray $\Gamma_\theta$ emerges from the origin 
and constitutes the angle $\theta\in(0,\pi/2)$ with the $z$-axis. 
The coordinate system $(s,t)$ is
associated with $\Gamma_\theta$.}
\label{fig}
\end{figure}

Note that because of the choice $\eps\in(0,\tan\theta)$ we have 
$\supp \omega_{n,p}\subset \dR^2_+$ for all $n\in\dN$ and, 
moreover, the distances between the $z$-axis and the supports of 
$\omega_{n,p}$ satisfy 
\begin{equation}
\label{rho}
\rho_n := 
\inf\{r\colon(r,z)\in\supp\omega_{n,p}\}\rightarrow +\infty,
\qquad n\rightarrow\infty.
\end{equation}
By dominated convergence,
using \eqref{chi1} and \eqref{chi2}, we get
\begin{equation}
\label{omeganplimit}
\begin{split}
\|\omega_{n,p}\|^2_{L^2(\dR^2_+)} & = 
\Bigg(\frac{1}{n}
\int_{n}^{2n}\big|\chi_1(\tfrac{s}{n})e^{{\rm i} ps}\big|^2 
{\rm d}s\Bigg)
\Bigg(
\int_{-\varepsilon n}^{\varepsilon n}
\big|\chi_2(\tfrac{t}{n})\big|^2e^{-\alpha|t|} 
{\rm d}t\Bigg)\\[0.2ex]
& = \int_{ -\eps n}^{ \eps n}
\big|\chi_2(\tfrac{t}{n})\big|^2e^{-\alpha|t|}{\rm d}t
\rightarrow \int_{-\infty}^\infty e^{-\alpha|t|}{\rm d}t = 
\frac{2}{\alpha},\quad n\rightarrow\infty.
\end{split}
\end{equation}
We denote by $\omega_{n,p,\pm}$ 
the restrictions of 
$\omega_{n,p}$ onto the open subsets 
\[
S_+ = \{(r,z)\in\dR^2_+\colon z > r\cot\theta\}
\quad\text{and}\quad
S_- = \{(r,z)\in\dR^2_+\colon z < r\cot\theta\}
\] 
of $\dR^2_+$.
The partial derivatives of $\omega_{n,p,\pm}$ 
with respect to $s$ and $t$ are given by
\begin{equation*}
\begin{split}
\partial_s\omega_{n,p,\pm} &
= 
\tfrac{1}{\sqrt{n}}\Big(\tfrac{1}{n}\chi_1'(\tfrac{s}{n})e^{{\rm i}ps}
+{\rm i}p\chi_1(\tfrac{s}{n})e^{{\rm i}ps}\Big)
\Big(\chi_2(\tfrac{t}{n})e^{\pm\frac{\alpha}{2}t}\Big),\\
\partial_t\omega_{n,p,\pm} &= \tfrac{1}{\sqrt{n}}
\Big(\chi_1(\tfrac{s}{n})e^{{\rm i}ps}\Big)
\Big(\tfrac{1}{n}\chi_2'(\tfrac{t}{n})e^{\pm\frac{\alpha}{2}t} \pm
\tfrac{\alpha}{2}\chi_2(\tfrac{t}{n})e^{\pm\frac{\alpha}{2}t}\Big).
\end{split}
\end{equation*}
Similarly as in \eqref{omeganplimit}, 
using dominated convergence, we get
\begin{equation}
\label{nablaomeganplimit}
\|\nabla\omega_{n,p}\|_{L^2(\dR^2_+;\dC^2)}^2 = 
\int_{\dR^2_+}
\big(|\partial_s \omega_{n,p}|^2+|\partial_t \omega_{n,p}|^2)\rmd s\rmd t
\rightarrow 
\Big(p^2 +\frac{\alpha^2}{4}\Big)\frac{2}{\alpha},
\quad n\rightarrow\infty.
\end{equation}
Let us define the sequence of functions 
$\psi_{n,p}\colon\dR^3\rightarrow\dC$ 
as
\begin{equation}
\label{psinp}
\psi_{n,p}(r,\varphi,z) :=
\frac{\omega_{n,p}(r,z)}{\sqrt{2\pi r}},
\qquad n\in\dN,
\end{equation}
where the functions $\omega_{n,p}\colon\dR^2_+\rightarrow\dC$ 
are interpreted as rotationally invariant functions on $\dR^3$ 
in the cylindrical coordinate system $(r,\varphi,z)$. 
The hypersurface $\cC_\theta$ separates the Euclidean space
$\dR^3$ into two unbounded Lipschitz domains 
$\Omega_+$ and $\Omega_-$, 
where
\[
\begin{split}
\Omega_+ & =  
\big\{ (x,y,z)\in\dR^3\colon z > \cot(\theta)\sqrt{x^2 + y^2}\big\},\\
\Omega_- & =  
\big\{ (x,y,z)\in\dR^3\colon z < \cot(\theta)\sqrt{x^2 + y^2}\big\}.
\end{split}
\]
We use the notation $\psi_{n,p,\pm} := \psi_{n,p}|_{\Omega_\pm}$. 
Then $\psi_{n,p,\pm}\in C^\infty(\Omega_\pm)$ and
from \eqref{omeganplimit} we obtain
\begin{equation}
\label{psinplimit}
\|\psi_{n,p}\|^2_{L^2(\dR^3)} = 
\|\omega_{n,p}\|^2_{L^2(\dR^2_+)}\rightarrow \frac{2}{\alpha},
\qquad n\rightarrow\infty.
\end{equation}
We claim that $\psi_{n,p}\in\dom(-\Delta_{\alpha,\cC_\theta})$. For this it remains to check that the boundary conditions
\begin{equation}\label{jussibcs}
 \psi_{n,p,+}|_{\Sigma} = \psi_{n,p,-}|_{\Sigma}\quad\text{and}\quad 
 \partial_{\nu_+}\psi_{n,p,+}|_{\Sigma} +
\partial_{\nu_-} \psi_{n,p,-}|_{\Sigma}=\alpha\psi_{n,p}|_{\Sigma}
\end{equation}
are satisfied; cf. \cite[Theorem 3.3\,(i)]{BEL13}. In fact,
by the definition of $\omega_{n,p}$ we have
$\omega_{n,p,+}|_{\Sigma} = \omega_{n,p,-}|_{\Sigma}$,
which implies that the first condition in \eqref{jussibcs} holds. 
 Furthermore, one computes
\begin{equation}\label{okusethis}
\partial_{\nu_+}\omega_{n,p,+}|_{\Sigma} +
\partial_{\nu_-} \omega_{n,p,-}|_{\Sigma} = \alpha \tfrac{1}{\sqrt{n}}
\Big(\chi_1(\tfrac{s}{n})\exp({\rm i} ps)\Big)=
\alpha\omega_{n,p}|_{\Sigma}.
\end{equation}
The gradient of $\psi_{n,p,\pm}$
can be expressed as
\begin{equation*}
\nabla \psi_{n,p,\pm} = \tfrac{1}{\sqrt{2\pi r}}\nabla\omega_{n,p,\pm}
+ \omega_{n,p,\pm}\nabla\big(\tfrac{1}{\sqrt{2\pi r}}\big),
\end{equation*}
where $\nabla$ acts on the functions $(r,\varphi,z)\mapsto \omega_{n,p,\pm}(r,z)$ and $(r,\varphi,z)\mapsto\tfrac{1}{\sqrt{2\pi r}}$.
Hence, we obtain
\[
\begin{split}
\partial_{\nu_+}\psi_{n,p,+}|_{\Sigma} +
\partial_{\nu_-} \psi_{n,p,-}|_{\Sigma} &= \big(\tfrac{1}{\sqrt{2\pi r}}\big|_{\Sigma}\big)\big(\partial_{\nu_+}\omega_{n,p,+}|_{\Sigma} +
\partial_{\nu_-} \omega_{n,p,-}|_{\Sigma}\big)\\
&\qquad\qquad+
\big(\omega_{n,p}|_\Sigma\big)
\big(\partial_{\nu_+}
\big(\tfrac{1}{\sqrt{2\pi r}}\big)\big|_{\Sigma}
+\partial_{\nu_-}
\big(\tfrac{1}{\sqrt{2\pi r}}\big)\big|_{\Sigma}\big)\\
&=\big(\tfrac{1}{\sqrt{2\pi r}}\big|_{\Sigma}\big)
\alpha\big(\omega_{n,p}|_{\Sigma}\big) =
\alpha\psi_{n,p}|_{\Sigma},
\end{split}
\]
where \eqref{okusethis} was used in the second equality.
Thus we have verified \eqref{jussibcs} and therefore
$\psi_{n,p}\in\dom(-\Delta_{\alpha,\cC_\theta})$.
Moreover, according to  \cite[Theorem 3.3\,(i)]{BEL13} we also have
\begin{equation}
\label{Lap}
-\Delta_{\alpha,\cC_\theta}\psi_{n,p} = (-\Delta\psi_{n,p,+})\oplus(-\Delta\psi_{n,p,-}).
\end{equation}
Using the expression for the
three-dimensional Laplacian in cylindrical coordinates
we find
\[
-\Delta\psi_{n,p,\pm} =
-\tfrac{1}{r}\partial_r(r\partial_r\psi_{n,p,\pm}) 
-\partial_z^2\psi_{n,p,\pm},
\]
where the angular term is absent since the functions $\psi_{n,p,\pm}$ 
do not depend on $\varphi$.
The above expression can be rewritten as
\begin{equation}
\label{Lap1}
-\Delta\psi_{n,p,\pm}  = 
-\partial_r^2\psi_{n,p,\pm} -\partial_z^2\psi_{n,p,\pm} 
- \tfrac{1}{r}(\partial_r\psi_{n,p,\pm}).
\end{equation}

Next we compute the first and second order
partial derivatives of $\psi_{n,p,\pm}$ 
with respect to $r$:
\begin{equation}
\label{r}
\begin{split}
\partial_r\psi_{n,p,\pm} &= 
\frac{\partial_r\omega_{n,p,\pm}}{\sqrt{2\pi r}} -
\frac{\omega_{n,p,\pm}}{2\sqrt{2\pi} r^{3/2}},\\
\partial_r^2\psi_{n,p,\pm} &= 
\frac{\partial_r^2\omega_{n,p,\pm}}{\sqrt{2\pi r}} -
\frac{\partial_r\omega_{n,p,\pm}}{\sqrt{2\pi} r^{3/2}} + 
\frac{3}{4}\frac{\omega_{n,p,\pm}}{\sqrt{2\pi}r^{5/2}}.
\end{split}
\end{equation}
The last two summands in the expression for
$\partial_r^2\psi_{n,p,\pm}$ can be estimated 
in $L^2$-norm as
\begin{equation}
\label{estimates}
\begin{split}
\Bigg\|
\frac{\partial_r\omega_{n,p,\pm}}{\sqrt{2\pi} r^{3/2}}
\Bigg\|^2_{L^2(\dR^3)}
&\le\frac{1}{\rho_n^2}\|\nabla\omega_{n,p}\|^2_{L^2(\dR^2_+;\dC^2)} \rightarrow 0,
\qquad n\rightarrow \infty,
\\[0.2ex]
\frac{9}{16}\Bigg\|\frac{\omega_{n,p,\pm}}
{\sqrt{2\pi}r^{5/2}}\Bigg\|_{L^2(\dR^3)}^2
& \le \frac{9}{16\rho_n^4}
\|\omega_{n,p}\|^2_{L^2(\dR^2_+)} \rightarrow 0,
\qquad n\rightarrow \infty,
\end{split}
\end{equation}
where we have used 
\eqref{rho}, \eqref{omeganplimit} and \eqref{nablaomeganplimit}.
The second order partial derivatives of $\psi_{n,p,\pm}$ with
respect to $z$ are 
\begin{equation}
\label{z}
\partial_z^2\psi_{n,p,\pm} = 
\frac{\partial_z^2\omega_{n,p,\pm}}{\sqrt{2\pi r}}.
\end{equation}
Using \eqref{r}, \eqref{estimates}, \eqref{z} 
and the invariance of the Laplacian
under rotation of the coordinate system we obtain that
\begin{equation}
\label{Lap2}
-\partial_r^2\psi_{n,p,\pm} - \partial_z^2\psi_{n,p,\pm} 
=
-\frac{1}{\sqrt{2\pi r}}
\big(\partial_s^2\omega_{n,p,\pm} +\partial_t^2\omega_{n,p,\pm}\big)
 + o(1), \quad n\rightarrow \infty;
\end{equation}
here and in the following 
we understand $o(1)$ in the strong sense with
respect to the corresponding $L^2$-norm.
With the help of \eqref{r} the norm of the last summand on the right hand side in \eqref{Lap1} 
can be estimated as
\begin{equation*}
\Bigg\|\frac{\partial_r\psi_{n,p,\pm}}{r}\Bigg\|_{L^2(\dR^3)}^2 
\le\Bigg\|\frac{\partial_r\omega_{n,p,\pm}}{\sqrt{2\pi}r^{3/2}}\Bigg\|_{L^2(\dR^3)}^2+
\Bigg\|\frac{\omega_{n,p,\pm}}{2\sqrt{2\pi}r^{5/2}}\Bigg\|_{L^2(\dR^3)}^2, 
\end{equation*}
and from \eqref{estimates} we conclude
\[
\Bigg\|\frac{\partial_r\psi_{n,p,\pm}}{r}\Bigg\|_{L^2(\dR^3)}^2 = o(1), \qquad n\rightarrow \infty.
\]

From \eqref{Lap1},
the latter result and \eqref{Lap2} we obtain
\begin{equation}
\label{Lap3}
-\Delta\psi_{n,p,\pm} = 
-\frac{1}{\sqrt{2\pi r}}
\big(\partial_s^2\omega_{n,p,\pm} +\partial_t^2\omega_{n,p,\pm}\big)
 + o(1),
\quad n\rightarrow \infty.
\end{equation}
Again using dominated convergence we compute
\begin{equation}
\label{s}
\begin{split}
\partial_s^2\omega_{n,p,\pm}  
& \!=\! \tfrac{1}{\sqrt{n}}
\Big(\chi_2(\tfrac{t}{n})e^{\pm\frac{\alpha}{2}t}\Big)\!
\Big(\tfrac{1}{n^2}\chi_1^{\prime\prime}(\tfrac{s}{n})e^{\rmi ps} 
+ \tfrac{2\rmi p}{n}\chi_1'(\tfrac{s}{n})e^{\rmi p s}  
- p^2\chi_1(\tfrac{s}{n})e^{\rmi  ps}\Big)\\
&= -p^2\omega_{n,p,\pm} + o(1),\qquad n\rightarrow \infty,
\end{split}
\end{equation}
and
\begin{equation}
\label{t}
\begin{split}
\partial_t^2\omega_{n,p,\pm} & 
\!=\! \tfrac{1}{\sqrt{n}}\Big(\chi_1(\tfrac{s}{n})e^{{\rm i}ps}\Big)\!
\Big(\tfrac{1}{n^2}\chi_2''(\tfrac{t}{n})e^{\pm\frac{\alpha}{2}t}
\!\pm\!
\tfrac{\alpha}{n}\chi_2'(\tfrac{t}{n})e^{\pm\frac{\alpha}{2}t}
\!+\!
\tfrac{\alpha^2}{4}\chi_2(\tfrac{t}{n})e^{\pm\frac{\alpha}{2}t}\Big)
\\
&= \tfrac{\alpha^2}{4}\omega_{n,p,\pm} + o(1),\qquad n\rightarrow \infty.
\end{split}
\end{equation}
Finally, employing \eqref{Lap}, \eqref{Lap3}, 
the definition of $\psi_{n,p}$ in \eqref{psinp} 
and \eqref{s}, \eqref{t} we arrive at 
\begin{equation}
\label{Lap4}
-\Delta_{\alpha,\cC_\theta}\psi_{n,p} = 
\Big(-\frac{\alpha^2}{4} +p^2\Big)\psi_{n,p}+o(1),
\qquad n\rightarrow \infty.
\end{equation}

Since the supports of $\psi_{ 2^k,p}$ and $\psi_{2^{k'},p}$, $k\not=k^\prime$, are disjoint the sequence $\{\psi_{2^k,p}\}_k$  converges weakly to zero.
Moreover, by  \eqref{psinplimit} we have $\liminf\|\psi_{ 2^k,p}\|_{L^2(\dR^3)} > 0$ 
and hence \eqref{Lap4} implies that $\{\psi_{2^k,p}\}_{k}$
is a singular sequence for the operator $-\Delta_{\alpha,\cC_\theta}$ corresponding to the point $-\alpha^2/4 +p^2$.
Therefore, $-\alpha^2/4+p^2\in\sess(-\Delta_{\alpha,\cC_\theta})$ for all $p\in\dR$ (see, e.g. \cite[Theorem 9.1.2]{BS87} or \cite[Proposition 8.11]{S})
and it follows that $[-\alpha^2/4,+\infty)\subseteq\sess(-\Delta_{\alpha,\cC_\theta})$.

\vspace{0.8ex}

\noindent\emph{Step 2.}
In this step we show the inclusion      
$\sess(-\Delta_{\alpha,\cC_\theta})\subseteq [-\alpha^2/4,+\infty)$
using form decomposition methods. 
For sufficiently large $n\in\dN$ we define three subsets of 
the closed half-plane
$\ov{\dR^2_+} := \{(r,z)\in\dR^2\colon r\ge 0,\,z\in\dR\}$
\begin{equation*}
\begin{split}
 \pi_{n}^1 &:= 
\{(r(s,t),z(s,t))\in\ov{\dR^2_+}\colon s > n, |t| < \sqrt{n}\}
\subset\ov{\dR^2_+},\\
\pi_{n}^2 & := 
\big\{(r(s,t),z(s,t))\in\ov{\dR^2_+}\colon s < n, |t| < \sqrt{n}\big\}
\subset\ov{\dR^2_+},\\
 \pi_n^3 & := 
\big\{(r(s,t),z(s,t))\in\ov{\dR^2_+}\colon |t| > \sqrt{n}\big\}
\subset\ov{\dR^2_+},
\end{split}
\end{equation*}
as shown in Figure~\ref{fig2}.

\begin{figure}[H]
\begin{center}
\begin{picture}(150,140)
\put(10,0){\vector(0,1){140}}
\put(10,70){\vector(1,0){70}}
\multiput(10,70)(12,18){4}{\line(2,3){10}}
\put(60,143){\vector(2,3){7}}
\multiput(10,70)(15,-10){4}{\line(3,-2){11}}
\put(70,30){\vector(3,-2){10}}
\put(10,90){\line(2,3){40}}
\put(10,50){\line(2,3){60}}
\put(20,105){\line(3,-2){18}}
\put(47,121){$\Gamma_\theta$}
\put(65,145){$s$}
\put(12,135){$z$}
\put(74,62){$r$}
\put(71,18){$t$}
\put(30,110){$\pi_n^1$}
\put(30,47){$\pi_n^3$}
\put(20,127){$\pi_n^3$}
\put(19,84){ $\pi_n^2$}
\put(11,77){{\small $\theta$}}
\qbezier(10,74)(11,76)(12,73)
\end{picture}
\end{center}
\caption{The subsets $ \pi_{n}^1$, $ \pi_{n}^2$ and 
$\pi_{n}^3$ of the closed half-plane $\ov{\dR^2_+}$.}
\label{fig2}
\end{figure}
\noindent The ray $\Gamma_\theta$, which emerges from the origin
and constitutes the angle $\theta$ with $z$-axis, 
is decomposed into 
\[
\begin{split}
\Gamma^1_{\theta,n} &:= \{((r(s,t),z(s,t))\in\Gamma_\theta\colon s > n\},\\
\Gamma^2_{\theta,n} &:= \{((r(s,t),z(s,t))\in\Gamma_\theta\colon s < n\}.
\end{split}
\]
The splitting $\{ \pi_n^k\}_{k=1}^3$ of $\ov{\dR^2_+}$
induces the splitting of $\dR^3$ into three domains
\[
\Omega_n^k := \big\{(r,\varphi,z)\colon (r,z)\in\pi_n^k,~ 
\varphi\in [0,2\pi)\big\}\subset\dR^3,\qquad k=1,2,3,
\]
and the splitting of the conical surface $\cC_\theta$
into two parts
\[
\begin{split}
\cC^1_{\theta,n} &:= 
\{(r,\varphi,z)\colon 
(r,z)\in\Gamma_{\theta,n}^1,~\varphi\in[0,2\pi)\}\subset\cC_\theta,\\
\cC^2_{\theta,n} &:= 
\{(r,\varphi,z)\colon 
(r,z)\in\Gamma_{\theta,n}^2,~\varphi\in[0,2\pi)\}\subset\cC_\theta.
\end{split}
\]
We agree to denote the restriction of $ \psi\in L^2(\dR^3)$ onto $\Omega_n^k$
with $k=1,2,3$ by $\psi_k$. 

Consider the quadratic form
\[
\begin{split}
\fra_{\alpha,\cC_\theta,n}[ \psi] &
:= \sum_{k=1}^3\|\nabla  \psi_k\|^2_{L^2(\Omega_n^k;\dC^3)}
- \alpha\| \psi_1|_{\cC^1_{\theta,n}}\|^2_{L^2(\cC^1_{\theta,n})} 
- \alpha\| \psi_2|_{\cC^2_{\theta,n}}\|^2_{L^2(\cC^2_{\theta,n})},\\
\dom\fra_{\alpha,\cC,n} & = \bigoplus_{k=1}^3 H^1(\Omega^k_n).
\end{split}
\]
As in the proof of 
\cite[Proposition 3.1]{BEL13} one verifies that the form $\fra_{\alpha,\cC_\theta,n}$
is closed, densely defined, symmetric and semibounded from below.
Hence $\fra_{\alpha,\cC_\theta,n}$ induces
a self-adjoint operator $-\Delta_{\alpha,\cC_\theta,n}$ in $L^2(\dR^3)$ via the first representation theorem \cite[Theorem VI.2.1]{Kato}.
The operator $-\Delta_{\alpha,\cC_\theta,n}$ can be decomposed
into an orthogonal sum $\oplus_{k=1}^3 H_{n,k}$ of self-adjoint operators $H_{n,k}$ in $L^2(\Omega_n^k)$ with
respect to the orthogonal decomposition 
$L^2(\dR^3) = \oplus_{k=1}^3 L^2(\Omega_n^k)$, where $H_{n,1}$ and $H_{n,2}$ correspond to the quadratic forms
\begin{equation*}
 \begin{split}
  \fra_{n,1}[ \psi_1] &= \|\nabla  \psi_1\|^2_{L^2(\Omega_n^1;\dC^3)} -
  \alpha
  \| \psi_1|_{\cC^1_{\theta,n}}\|^2_{L^2(\cC^1_{\theta,n})},\qquad 
  \dom\fra_{n,1} = H^1(\Omega_n^1),\\
   \fra_{n,2}[ \psi_2] &= 
   \|\nabla  \psi_2\|^2_{L^2(\Omega_n^2;\dC^3)} -
   \alpha\| \psi_2|_{\cC^2_{\theta,n}}\|^2_{L^2(\cC^2_{\theta,n})},\qquad \dom\fra_{n,2} = H^1(\Omega_n^2),
 \end{split}
\end{equation*}
respectively, 
and $H_{n,3}$ corresponds to the quadratic form 
$\fra_{n,3}[ \psi_3]=
\|\nabla  \psi_3\|^2_{L^2(\Omega_n^{3};\dC^3)}$, $\dom\fra_{n,3} = H^1(\Omega_n^3)$.

Let us first estimate the spectrum of $H_{n,1}$. For this note that $C^\infty(\Omega_{n}^1)\cap H^1(\Omega_n^1)$ is a core of $\fra_{n,1}$ and thus it suffices 
to use functions from this set in the estimates below (see, e.g. \cite[Theorem 4.5.3]{D95}).
For any $ \psi_1\in  C^\infty(\Omega_{n}^1)\cap H^1(\Omega_n^1)$ normalized as
$\| \psi_1\|_{L^2(\Omega_{n}^1)}=1$ we obtain 
\begin{equation*}
 \begin{split}
 \fra_{n,1}[ \psi_1]  \ge 
\int_0^{2\pi}\Bigg(\int_n^{+\infty}\int_{-\sqrt{n}}^{\sqrt{n}} &\,
r(s,t)|\partial_t \psi_1(s,t,\varphi)|^2
\rmd t \rmd s \\ & \qquad - \alpha \int_n^{+\infty} 
r(s,0)| \psi_1(s,0,\varphi)|^2 
\rmd s\Bigg) \rmd \varphi, 
 \end{split}
\end{equation*}
where we have used the form of the gradient in cylindrical coordinates 
and the invariance of the gradient with respect to rotations of the coordinate system, and the nonnegative terms corresponding to the partial derivatives of $\psi_1$
with respect to $\varphi$ and $s$ where estimated from below by zero.  
Note that for simple geometric reasons we have $r(s,t) \ge r(s,-\sqrt{n})$ for all $(s,t)\in \pi_n^1$. Using this observation we get
\begin{equation}\label{fran1}
\begin{split}
\fra_{n,1}[ \psi_1] \ge
\int_0^{2\pi}\int_n^{+\infty}
r(s,-\sqrt{n})\Bigg(\int_{-\sqrt{n}}^{\sqrt{n}}&
|\partial_t \psi_1(s,t,\varphi)|^2 \rmd t \\
&\qquad -\tfrac{\alpha r(s,0)}{r(s,-\sqrt{n})}| \psi_1(s,0,\varphi)|^2 \Bigg)
\rmd s \rmd \varphi.
\end{split}
\end{equation}
Consider the closed, densely defined, symmetric and semibounded form
\begin{equation*}
 \mathfrak b[h]=\int_{-\sqrt{n}}^{\sqrt{n}} \vert h^\prime(t)\vert^2\rmd t - \beta \vert h(0)\vert ^2,\qquad \dom\mathfrak b=H^1((-\sqrt{n},\sqrt{n})),
\end{equation*}
and denote by $ \mu(\beta, 2\sqrt{n}) < 0$
the lower bound of the spectrum
of the associated 1-D Schr\"odinger operator on the interval $(-\sqrt{n},\sqrt{n})$
with Neumann boundary conditions at the endpoints and attractive
$\delta$-interaction of strength $\beta > 0$ located at $0$. Then 
\begin{equation*}
 \mathfrak b[h]\geq  \mu(\beta, 2\sqrt{n}) \int_{-\sqrt{n}}^{\sqrt{n}} \vert h (t)\vert^2\rmd t 
\end{equation*}
holds for all $h\in H^1((-\sqrt{n},\sqrt{n}))$ and hence \eqref{fran1} can be further estimated as 
\begin{equation}
\label{fran2}
\fra_{n,1}[ \psi_1]\ge 
\int_0^{2\pi}\int_n^{+\infty}
 \mu\Big(\tfrac{\alpha r(s,0)}{r(s,-\sqrt{n})},2\sqrt{n}\Big)
\int_{-\sqrt{n}}^{\sqrt{n}}r(s,-\sqrt{n})| \psi_1(s,t,\varphi)|^2
\rmd t \rmd s \rmd \varphi.
\end{equation}
By the definition of $ \pi_n^1$ one has 
\begin{equation}
\label{frac}
r(s,-\sqrt{n}) = r(s,t)\Big(1 + \cO(\tfrac{1}{\sqrt{n}})\Big),
\qquad n\rightarrow\infty,
\end{equation}
for $(s,t)\in  \pi_n^1$,
where the remainder is uniform in $s$.
Hence, we obtain from \eqref{fran2} and \eqref{frac}
\begin{equation}
\label{fran3}
\fra_{n,1}[ \psi_1] \ge
 \mu\Big(\alpha\big(1+\cO\big(\tfrac{1}{\sqrt{n}}\big)\big),2\sqrt{n}\Big)
\Big(1 + \cO(\tfrac{1}{\sqrt{n}})\Big),
\qquad n\rightarrow\infty,
\end{equation}
where we used that
\[
\int_0^{2\pi}\int_n^{+\infty}\int_{-\sqrt{n}}^{\sqrt{n}}
r(s,t)| \psi_1(s,t,\varphi)|^2\rmd t\rmd s \rmd \varphi 
= \| \psi\\
_1\|^2_{L^2(\Omega_n^1)} = 1.
\]
According to \cite[Proposition 2.5]{EY02} the following estimate 
\[
 \mu(\beta, 2\sqrt{n}) \ge -\tfrac{\beta^2}{4} - C\beta^2\exp(-\tfrac12\beta \sqrt{n})
\]
holds with some constant $C > 0$ and $n$ sufficiently large.
Hence,
\[
 \mu\Big(\alpha\big(1+O\big(\tfrac{1}{\sqrt{n}}\big)\big),2\sqrt{n}\Big) 
\ge -\tfrac{\alpha^2}{4} + \cO(\tfrac{1}{\sqrt{n}}),
\qquad n\rightarrow\infty.
\]
Plugging the above estimate into \eqref{fran3} we arrive at
\[
\fra_{n,1}[ \psi_1] \ge 
-\tfrac{\alpha^2}{4} + \cO(\tfrac{1}{\sqrt{n}}),
\qquad n\rightarrow\infty.
\]
Hence, for any $\eps > 0$ there exists a sufficiently large $n$
for which 
\begin{equation}
\label{spec}
\inf\sigma(H_{n,1}) \ge -\tfrac{\alpha^2}{4} -\eps.
\end{equation}

As $H^1(\Omega_n^2)$ is compactly embedded into $L^2(\Omega_n^2)$
the essential spectrum of $H_{n,2}$ is empty. 
The operator $H_{n,3}$ is non-negative and hence $\sigma(H_{n,3})\subseteq [0,+\infty)$. 
Due to the orthogonal decomposition 
$-\Delta_{\alpha,\cC_\theta,n}=\oplus_{k=1}^3 H_{n,k}$
the property \eqref{spec} implies  that for any $\eps > 0$
there exists a sufficiently large $n$ for which
\begin{equation}
\label{estn}
\inf\sess(-\Delta_{\alpha,\cC_\theta,n}) \ge-\tfrac{\alpha^2}{4} -\eps.
\end{equation}

Finally, we apply a Neumann bracketing argument.
Notice that the ordering
$\fra_{\alpha,\cC_\theta,n}\leq\fra_{\alpha,\cC_\theta}$
holds in the sense of quadratic forms; cf. \cite[\S VI.5]{Kato}. 
Hence by 
\cite[Theorem 10.2.4]{BS87} 
\begin{equation}
\label{ordering}
\inf\sess(-\Delta_{\alpha,\cC_{\theta},n})
\le
\inf\sess(-\Delta_{\alpha,\cC_\theta}).
\end{equation}
In view of \eqref{ordering} the  estimate \eqref{estn}
implies that for any $\eps > 0$
\[
\inf\sess(-\Delta_{\alpha,\cC_\theta}) \ge -\tfrac{\alpha^2}{4}-\eps
\]
and thus passing to the limit $\eps \rightarrow 0+$ we arrive at
\[
\inf\sess(-\Delta_{\alpha,\cC_\theta}) \ge -\tfrac{\alpha^2}{4},
\]
which shows the inclusion 
$\sess(-\Delta_{\alpha,\cC_\theta})\subseteq [-\alpha^2/4,+\infty)$
and finishes the proof of Theorem~\ref{mainthm1}.
\end{proof}

\section{Discrete spectrum of $-\Delta_{\alpha,\cC_\theta}$}

In this section we show that the discrete spectrum of the 
self-adjoint operator
$-\Delta_{\alpha,\cC_\theta}$ below the bottom $-\alpha^2/4$ of the essential spectrum  is infinite 
for all angles  $\theta\in(0,\pi/2)$
and we estimate the rate of the convergence of these eigenvalues to $-\alpha^2/4$
with the help of  variational principles.
The following lemma will be useful.
\begin{lem}
\label{lem:sep}
Let $\fra_{\alpha,\cC_\theta}$ be the form in \eqref{fra}.
For $\omega\in H^1(\dR^2_+)$ with compact support
$\supp \omega \subset\dR^2_+$ 
define the function
$\psi(r,\varphi,z) := \tfrac{\omega(r,z)}{\sqrt{2\pi r}}$.
Then $\psi \in H^1(\dR^3)$ and
\begin{equation}\label{bitteschoen}
\fra_{\alpha,\cC_\theta}[\psi] 
= \|\nabla\omega\|^2_{L^2(\dR^2_+;\dC^2)} 
- \int_{\dR^2_+}\frac{1}{4r^2}|\omega(r,z)|^2\rmd r\rmd z 
-\alpha\|\omega|_{\Gamma_\theta}\|^2_{L^2(\Gamma_\theta)},
\end{equation}
where $\Gamma_\theta$ is the ray in Figure~\ref{fig}. 
\end{lem}

\begin{proof}
First of all observe that 
\begin{equation}
\label{psi1}
\|\psi\|^2_{L^2(\dR^3)} = \int_\dR\int_{\dR_+}\int_0^{2\pi} \tfrac{\vert\omega(r,z)\vert^2}{2\pi r}\,r\,\rmd \varphi\rmd r\rmd z= \|\omega\|^2_{L^2(\dR^2_+)} <\infty.
\end{equation}
Moreover, we compute
\begin{equation}
\label{deriv}
\partial_r \psi = \frac{\partial_r\omega}{\sqrt{2\pi r}} 
- \frac{\omega}{2r\sqrt{2\pi r}}\quad\text{and}\quad 
\partial_z\psi = \frac{\partial_z\omega}{\sqrt{2\pi r}},
\end{equation}
and setting $\rho := \inf\{r \colon (r,z)\in\supp\omega\}>0$ we obtain
\begin{equation}
\label{psi2}
\begin{split}
\|\nabla\psi\|^2_{L^2(\dR^3)} &= 
\|\partial_r\psi\|^2_{L^2(\dR^3)}+ 
\|\partial_z\psi\|^2_{L^2(\dR^3)} \\
&\leq
2
\big\|\tfrac{\partial_r\omega}{\sqrt{2\pi r}} \big\|_{L^2(\dR^3)}^{ 2}+
 2\big\|\tfrac{\omega}{2r\sqrt{2\pi r}}\big\|_{L^2(\dR^3)}^{2}+ 
\big\|\tfrac{\partial_z\omega}{\sqrt{2\pi r}}
\big\|^2_{L^2(\dR^3)} \\
&\leq 2\|\partial_r\omega\|_{L^2(\dR^2_+)}^{2} + 
 \tfrac{1}{2\rho^2}\|\omega\|_{L^2(\dR^2_+)}^{ 2} +
\|\partial_z\omega\|^2_{L^2(\dR^2_+)} <\infty.
\end{split}
\end{equation}
Hence \eqref{psi1} and \eqref{psi2} imply $\psi \in H^1(\dR^3)$. Next we substitute $\psi$ in the form  $\fra_{\alpha,\cC_\theta}$
in \eqref{fra}.
It follows from the form of $\partial_z \psi$ in \eqref{deriv} and $\|\psi|_{\cC_\theta}\|^2_{L^2(\cC_\theta)}=\|\omega|_{\Gamma_\theta}\|^2_{L^2(\Gamma_\theta)}$ that
\begin{equation}\label{fraas}
\begin{split}
\fra_{\alpha,\cC_\theta}[\psi]
& = \int_{\dR}\int_{\dR_+}|\partial_r\psi|^22\pi r\rmd r \rmd z + \int_{\dR}\int_{\dR_+}|\partial_z\psi|^22\pi r\rmd r\rmd z - 
    \alpha \|\psi|_{\cC_\theta}\|^2_{L^2(\cC_\theta)}\\
    & = \int_{\dR}\int_{\dR_+}|\partial_r\psi|^22\pi r\rmd r \rmd z + \int_\dR\int_{\dR_+}|\partial_z\omega|^2\rmd r \rmd z - 
    \alpha \|\omega|_{\Gamma_\theta}\|^2_{L^2(\Gamma_\theta)}.
\end{split}
\end{equation}
Denote the first integral by $I_\psi$.
Making use of $\partial_r \psi$ in \eqref{deriv} we rewrite  $I_\psi$  as
\begin{equation}
\label{I1}
I_\psi = \int_{\dR}\int_{\dR_+}|\partial_r\omega|^2\rmd r \rmd z + \int_{\dR}\int_{\dR_+} \frac{1}{4r^2}|\omega|^2\rmd r \rmd z -\int_{\dR}\int_{\dR_+} 
\frac{1}{r}\RE(\partial_r\omega \ov\omega)
\rmd r\rmd z
\end{equation}
and the last term can be further rewritten as
\begin{equation}
\label{I2}
\int_{\dR}\int_{\dR_+} 
\frac{1}{r}\RE(\partial_r\omega \ov\omega)\rmd r \rmd z = \int_{\dR}\int_{\dR_+} 
\frac{1}{2r}\partial_r\big(|\omega|^2\big)\rmd r \rmd z = 
\int_{\dR}\int_{\dR_+} 
\frac{1}{2r^2}|\omega|^2\rmd  r\rmd z,
\end{equation}
where we integrated by parts and used the fact that $\supp \omega$ is contained in the open half-plane $\dR^2_+$.
Hence, \eqref{I1} and \eqref{I2} imply
\begin{equation*}
I_\psi = \int_\dR \int_{\dR_+}|\partial_r\omega|^2\rmd r \rmd z -
\int_\dR \int_{\dR_+}\frac{1}{4r^2}|\omega|^2\rmd r \rmd z.
\end{equation*}
Substituting this expression for the first integral in \eqref{fraas} we obtain \eqref{bitteschoen}.
\end{proof}

Now we are ready to formulate and prove our main result on the infiniteness of the discrete spectrum of $-\Delta_{\alpha,\cC_\theta}$ below the bottom
of the essential spectrum for all  $\alpha>0$ and $\theta \in (0,\pi/2)$. Recall that $-\Delta_{\alpha,\cC_\theta}$ is bounded from below, 
and hence it also follows that the 
discrete spectrum  has a single accumulation point, namely 
$-\alpha^2/4$.
\begin{thm}
\label{thm:main}
Let $-\Delta_{\alpha,\cC_\theta}$ be the self-adjoint operator in $L^2(\dR^3)$ associated to the form \eqref{fra}
and let  $\alpha>0$ and $\theta \in (0,\pi/2 )$. Then the discrete spectrum of 
$-\Delta_{\alpha,\cC_\theta}$ below $-\alpha^2/4$ is infinite, accumulates at $-\alpha^2/4$, and the eigenvalues 
$\lambda_k<-\alpha^2/4$  (enumerated in non-decreasing order with multiplicities taken into account) satisfy 
the estimate 
\begin{equation}\label{specest}
\lambda_k \le -\frac{\alpha^2}{4} - 
\frac{\gamma(\theta)}{n_k^4},\qquad k\in\dN,
\end{equation}
holds, where $\gamma(\theta) > 0$, $n_{k+1} := n_k^2 +n_k$ for  $k\in\dN$, and $n_1=N$ with $N\in\dN$ sufficiently large.
\end{thm}

\begin{proof}
Let us pick a function 
$\chi_1\in H_0^1(0,1)$ with 
$\|\chi_1\|_{L^2(0,1)} = 1$ such that 
\begin{equation}
\label{BM}
\|\chi_1'\|^2_{L^2(0,1)} < 
\frac{1}{4\sin^2\theta}\int_0^1 \frac{|\chi_1(t)|^2}{t^2}\rmd t
\end{equation}
holds; cf. \cite[Lemma in \S 1]{BM97}.
Let us fix $\eps > 0$ and choose 
$\chi_2\in C^\infty_0(-\eps,\eps)$ such that $0 \le \chi_2\le 1$ and 
$\chi_2(t) = 1$ for $|t| \le \eps/2$.
In the coordinate system $(s,t)$ in Figure~\ref{fig} 
we define
the sequence of functions
\[
\omega_n(s,t) := 
\tfrac{1}{n}\chi_1(\tfrac{s-n}{n^2})
\chi_2(\tfrac{t}{\sqrt{n}})
\exp(-\tfrac{\alpha}{2}|t|) \in H^1_0(\dR^2_+).
\]
For sufficiently large $n\in\dN$ the functions $\omega_n$
satisfy the conditions of Lemma~\ref{lem:sep}. The function $\omega_n$ can also be viewed as a function 
in $r$ and $z$; cf. Figure~\ref{fig}. 
Then we define
\begin{equation}
\label{psi}
\psi_n(r,\varphi,z) := \frac{\omega_n(r,z)}{\sqrt{2\pi r}}, \qquad n\in\dN.
\end{equation}
Using Lemma~\ref{lem:sep} we compute the values
\begin{equation}\label{Sn}
\begin{split}
S_n :=& \fra_{\alpha,\cC_\theta}[\psi_n] + 
\frac{\alpha^2}{4}\|\psi_n\|^2_{L^2(\dR^3)}\\
=& \|\nabla\omega_n\|^2_{L^2(\dR^2_+;\dC^2)} 
- \int_{\dR^2_+}\frac{1}{4r^2}|\omega_n|^2\rmd r \rmd z 
-\alpha\|\omega_n|_{\Gamma_\theta}\|^2_{L^2(\Gamma_\theta)}+\frac{\alpha^2}{4}\|\omega_n\|^2_{L^2(\dR^2_+)}.
\end{split}
\end{equation}
It is not difficult to check the asymptotics 
\begin{align}\label{asymp1}
&\int_{-\eps\sqrt{n}}^{\eps\sqrt{n}} 
\big|\chi_2\big(\tfrac{t}{\sqrt{n}}\big)\big|^2
e^{-\alpha|t|}\rmd t 
= \frac{2}{\alpha} + \cO(e^{-c\sqrt{n}}),&\qquad n\rightarrow\infty,\\
&\int_{-\eps\sqrt{n}}^{\eps\sqrt{n}} 
\big|\chi_2'(\tfrac{t}{\sqrt{n}}\big)|^2
e^{-\alpha|t|}\rmd t =\cO(e^{-c\sqrt{n}}),
&\qquad n\rightarrow\infty,\label{asymp2}\\\
&\int_{-\eps\sqrt{n}}^{\eps\sqrt{n}} 
\chi_2\big(\tfrac{t}{\sqrt{n}}\big)\chi_2'\big(\tfrac{t}{\sqrt{n}}\big)
e^{-\alpha|t|}\rmd t =\cO(e^{-c\sqrt{n}}),
&\qquad n\rightarrow\infty,\label{asymp3}
\end{align}
with some constant $c >0$. 
Using  \eqref{asymp1} 
we get 
\begin{equation}\label{norm}
\begin{split}
\frac{\alpha^2}{4}\|\omega_n\|^2_{L^2(\dR^2_+)}& = 
\frac{\alpha^2}{4}\Bigg(\frac{1}{n^2}\int_n^{n+n^2}|\chi_1(\tfrac{s-n}{n^2})|^2\rmd s\Bigg)
\Bigg(\int_{-\eps\sqrt{n}}^{\eps\sqrt{n}}
\big|\chi_2(\tfrac{t}{\sqrt{n}})\big|^2e^{-\alpha|t|}\rmd t\Bigg)\\[0.4ex]
&=\frac{\alpha}{2} + \cO(e^{-c\sqrt{n}}),\qquad n\rightarrow\infty,
\end{split}
\end{equation}
and
\begin{equation*}
\begin{split}
\|\partial_s\omega_n\|^2_{L^2(\dR^2_+)} &
\!=\! \Bigg(
\int_{-\eps\sqrt{n}}^{\eps\sqrt{n}}
\big|\chi_2(\tfrac{t}{\sqrt{n}})\big|^2e^{-\alpha|t|}\rmd t
\Bigg)\Bigg(
\frac{1}{n^4}\frac{1}{n^2}
\int_{n}^{n+n^2}
\big|\chi_1'(\tfrac{s-n}{n^2})\big|^2\rmd s\Bigg)\\[0.4ex]
& = \frac{2}{\alpha}\frac{1}{n^4}\|\chi_1'\|^2_{L^2(0,1)} 
+ \cO(e^{-c\sqrt{n}}),\qquad n\rightarrow\infty,
\end{split}
\end{equation*}
and from \eqref{asymp2} and \eqref{asymp3} we obtain 
\begin{equation*}
\begin{split}
\|\partial_t\omega_n\|^2_{L^2(\dR^2_+\!)} &\! = \!
\Bigg(\!\frac{1}{n^2}\int_n^{n+n^2}\!
\big|\chi_1(\tfrac{s-n}{n^2})\big|^2\rmd s\Bigg)\!\!
\Bigg(
\int_{-\eps\sqrt{n}}^{\eps\sqrt{n}}
\Bigg|\tfrac{\chi_2'\big(\tfrac{t}{\sqrt{n}}\big)}{\sqrt{n}} \!
-\! \tfrac{\alpha\sign (t)\chi_2\big(\tfrac{t}{\sqrt{n}}\big)}{2}\Bigg|^2\!
e^{-\alpha|t|}\rmd t\Bigg) \\[0.4ex]
&= \frac{\alpha}{2} + \cO(e^{-c\sqrt{n}}),\qquad n\rightarrow \infty,
\end{split}
\end{equation*}
that is, 
\begin{equation}\label{jussiincluded}
 \|\nabla\omega_n\|^2_{L^2(\dR^2_+;\dC^2)}=\frac{2}{\alpha}\frac{1}{n^4}\|\chi_1'\|^2_{L^2(0,1)} +\frac{\alpha}{2} + \cO(e^{-c\sqrt{n}}),\qquad n\rightarrow \infty.
\end{equation}
It is simple to see that
\begin{equation}\label{normtrace}
\alpha\|\omega_n|_{\Gamma_\theta}\|^2_{L^2(\Gamma_\theta)}  = 
\frac{\alpha}{n^2}\int_n^{n+n^2}\Big|\chi_1(\tfrac{s-n}{n^2})\Big|^2
\rmd s = 
\alpha\|\chi_1\|^2_{L^2(0,1)} = \alpha,
\end{equation}
and hence
it remains to estimate the term 
$\int_{\dR^2_+} \frac{1}{4r^2}|\omega_n|^2$ in \eqref{Sn}.
For that we make the following splitting
\begin{equation}
\label{lastterm}
\int_{\dR^2_+}\frac{1}{4r^2}|\omega_n(r,z)|^2drdz  = 
\int_{-\eps\sqrt{n}}^{\eps\sqrt{n}}\int_n^{n+n^2}
\frac{1}{4r(s,t)^2}|\omega_n(s,t)|^2\rmd s \rmd t =I_n + J_n,
\end{equation}
where
\begin{equation}
\label{In}
I_n  :=
\int_{-\eps\sqrt{n}}^{\eps\sqrt{n}}\int_n^{n+n^2}
\frac{1}{4r(s,0)^2}|\omega_n(s,t)|^2\rmd s \rmd t
\end{equation}
and
\begin{equation*}
J_n  := \int_{-\eps\sqrt{n}}^{\eps\sqrt{n}}\int_n^{n+n^2}
\Bigg(\frac{1}{4r(s,t)^2}-\frac{1}{4r(s,0)^2}\Bigg)
|\omega_n(s,t)|^2\rmd s \rmd t.
\end{equation*}
The term $J_n$ can be further rewritten as
\begin{equation}
\label{Jn2}
J_n = \int_{-\eps\sqrt{n}}^{\eps\sqrt{n}}\int_n^{n+n^2}
\frac{(r(s,0) - r(s,t))(r(s,0) +r(s,t))}{4r(s,t)^2r(s,0)^2}
|\omega_n(s,t)|^2\rmd s \rmd t.
\end{equation}
For geometric reasons we have 
$|r(s,0) - r(s,t)| \le a\sqrt{n}$ with some $0<a \leq\varepsilon$  and 
$r(s,t) > bn$ with some $b > 0$
for all $(s,t)\in\supp \omega_n$.
We first conclude from \eqref{Jn2} that
\[
|J_n| \le a\sqrt{n}
\int_{-\eps\sqrt{n}}^{\eps\sqrt{n}}\int_n^{n+n^2}
\Bigg|\frac{2}{4r(s,t)r(s,0)^2} + 
\frac{r(s,0) -r(s,t)}{4r(s,t)^2r(s,0)^2}\Bigg||\omega_n(s,t)|^2\rmd s \rmd t
\]
and hence 
\begin{equation}
\label{Jn3}
|J_n| \le\Big(\frac{2a}{b\sqrt{n}} + \frac{a^2}{b^2n}\Big)I_n
\end{equation}
follows together with \eqref{In}.
For $I_n$ we have
\begin{equation}
\label{In2}
\begin{split}
I_n & = \Bigg(\frac{1}{n^2}\int_n^{n+n^2}
\frac{|\chi_1(\tfrac{s-n}{n^2})|^2}{4s^2\sin^2\theta}\rmd s\Bigg)
\Bigg(\int_{-\eps\sqrt{n}}^{\eps\sqrt{n}}
\big|\chi_2\big(\tfrac{t}{\sqrt{n}}\big)\big|^2
e^{-\alpha|t|}\rmd t\Bigg) \\
&= \Bigg(
\frac{1}{n^4}
\int_0^1\frac{|\chi_1(u)|^2}{4\sin^2(\theta)(u+1/n)^2}\rmd u\Bigg) 
\Bigg(\int_{-\eps\sqrt{n}}^{\eps\sqrt{n}}
\big|\chi_2\big(\tfrac{t}{\sqrt{n}}\big)\big|^2
e^{-\alpha|t|}\rmd t\Bigg),
\end{split}
\end{equation}
and the choice of $\chi_1$ (see \eqref{BM}) together with monotone convergence yields
\begin{equation*}
\int_0^1\frac{|\chi_1(u)|^2}{(u+1/n)^2}\rmd u = 
\int_0^1\frac{|\chi_1(u)|^2}{u^2}\rmd u + o(1),\qquad n\rightarrow\infty.
\end{equation*}
Hence we conclude from \eqref{asymp1} and \eqref{In2} that
\begin{equation*}
 I_n = \frac{2}{\alpha}  \frac{1}{n^4} \frac{1}{4 \sin^2(\theta)} \int_0^1\frac{|\chi_1(u)|^2}{u^2}\rmd u + o\Bigg(\frac{1}{n^4}\Bigg),\qquad n\rightarrow\infty,
\end{equation*}
and from \eqref{Jn3} we find
\begin{equation*}
J_n=o\Bigg(\frac{1}{n^4}\Bigg),\qquad n\rightarrow\infty.
\end{equation*}
It follows that \eqref{lastterm} becomes 
\begin{equation}\label{lastterm2}
\int_{\dR^2_+}\frac{1}{4r^2}|\omega_n(r,z)|^2\rmd r \rmd z = 
\frac{2}{\alpha}\frac{1}{n^4} \frac{1}{4\sin^2(\theta)} 
\int_0^1\frac{|\chi_1(u)|^2}{ u^2}\rmd u
+ o\Bigg(\frac{1}{n^4}\Bigg),\qquad n\rightarrow\infty.
\end{equation}
Finally, \eqref{norm}, \eqref{jussiincluded}, \eqref{normtrace} and \eqref{lastterm2} yield 
\begin{equation}
\label{Snsymp}
S_n = \frac{2}{\alpha}\frac{1}{n^4}
\Bigg(
\|\chi_1'\|^2_{L^2(0,1)}
- \int_0^1\frac{|\chi_1(u)|^2}{4\sin^2(\theta) u^2}\rmd u
\Bigg) + o\Bigg(\frac{1}{n^4}\Bigg),\qquad n\rightarrow\infty,
\end{equation}
for $S_n$ in \eqref{Sn}.
In view of the above asymptotics and according to \eqref{BM} 
there exists $N \in\dN$  such that for all $n\ge N$  we have
\begin{equation}
\label{Snest}
S_{n} \le -\frac{2\gamma(\theta)}{\alpha n^4}
\end{equation}
for some constant $\gamma(\theta) > 0$.
Let us consider a sequence  $\{n_k\}_k$, 
where $n_1 := N$ and $n_{k+1} := n_k^2 + n_k$  for $k\in\dN$.
Then the functions $\psi_{n_k}$ in \eqref{psi}
have disjoint supports for all $k\in\dN$ and hence are orthogonal in $L^2(\dR^3)$.
The space
\[
F_k := {\rm span}\big\{\psi_{n_1},\psi_{n_2},\dots,\psi_{n_k}\big\}\subset H^1(\dR^3),
\]
has dimension $k$ and for an arbitrary
$\psi = \sum_{l=1}^k a_l \psi_{n_l}\in F_k$,  $a_l\in\dC$, we get
\begin{equation}
\label{normpsi}
\|\psi\|^2_{L^2(\dR^3)}  = 
\sum_{l=1}^k |a_l|^2\|\psi_{n_l}\|^2_{L^2(\dR^3)}
=\sum_{l=1}^k |a_l|^2\|\omega_{n_l}\|^2_{L^2(\dR^2_+)}
\leq \frac{2}{\alpha}
\sum_{l=1}^k |a_l|^2,
\end{equation}
where we have also used the estimate
$\|\omega_{n_l}\|^2_{L^2(\dR^2_+)}\leq \frac{2}{\alpha}$. 
Employing \eqref{Snest} we obtain
\[
\fra_{\alpha,\cC_\theta}[\psi] + 
\frac{\alpha^2}{4}\|\psi\|^2_{L^2(\dR^3)}  = 
\sum_{l=1}^k |a_l|^2 S_{n_l} \le 
-\frac{2\gamma(\theta)}{\alpha n_k^4} 
\sum_{l=1}^k |a_l|^2,
\]
where we have again used the disjointness of the supports of $\{\psi_{n_l}\}_{l=1}^k$.
Combining the above estimate with \eqref{normpsi} we get
\begin{equation}
\label{finalest}
\frac{\fra_{\alpha,\cC_\theta}[\psi]}{\|\psi\|^2_{L^2(\dR^3)}}
= -\frac{\alpha^2}{4} + 
\frac{\fra_{\alpha,\cC_\theta}[\psi] +
(\alpha^2/4)\|\psi\|^2_{L^2(\dR^3)}}{\|\psi\|^2_{L^2(\dR^3)}} \le 
-\frac{\alpha^2}{4} -\frac{\gamma(\theta)}{n_k^4}
 < -\frac{\alpha^2}{4}.
\end{equation}
Hence, according to \cite[Theorem 10.2.3]{BS87} the operator 
$-\Delta_{\alpha,\cC_\theta}$ has at least
$k$ eigenvalues below the bottom of the essential spectrum $-\alpha^2/4$. 
The above construction works for any $k\in\dN$, so that
the operator $-\Delta_{\alpha,\cC_\theta}$ has infinitely
many eigenvalues below $-\alpha^2/4$. The eigenvalue estimate 
\eqref{specest} follows from \cite[Theorem 10.2.3]{BS87} and \eqref{finalest}.
\end{proof}

Let $\theta\in (0,\pi/2)$ and $\cC_\theta$ be the conical surface as above.
A hypersurface $\Sigma\subset\dR^3$, which 
for some compact set $K\subset\dR^3$ satisfies the condition $\Sigma\setminus K = \cC_\theta \setminus K$ and which
splits the space $\dR^3$ into
two unbounded Lipschitz domains, is called
{\it a local
deformation} of $\cC_\theta$;
cf. \cite[Section 4.2]{BEL13}. Below we consider the self-adjoint Schr\"{o}dinger operator $-\Delta_{\alpha,\Sigma}$ 
with an attractive $\delta$-interaction of constant strength $\alpha >0$ 
supported on the Lipschitz hypersurface $\Sigma$. This Schr\"{o}dinger operator is defined via the quadratic form
\begin{equation}
\label{fra2}
\fra_{\alpha,\Sigma}[\psi] = 
\|\nabla \psi\|^2_{L^2(\dR^3;\dC^3)} - 
\alpha \int_{ \Sigma}\,\vert\psi\vert^2 \,\drm\sigma   
\qquad \dom\fra_{\alpha, \Sigma}  = H^1(\dR^3).
\end{equation}
The assertion on the essential spectrum in  the next theorem is a consequence of \cite[Theorem 4.7]{BEL13}; the infiniteness of the discrete spectrum can be shown as 
in the proof of Theorem~\ref{thm:main} using the same functions $\psi_n$ in \eqref{psi} and $n\in\dN$ sufficiently large.

\begin{thm}\label{locdeform}
Let $\theta\in (0,\pi/2)$ and $\alpha>0$.
Let $\Sigma$ be a local deformation of the cone $\cC_\theta$ and let $-\Delta_{\alpha,\Sigma}$ be the self-adjoint operator
in $L^2(\dR^3)$ associated to \eqref{fra2}. Then
\[
\sess(-\Delta_{\alpha,\Sigma}) = [-\alpha^2/4,+\infty),
\]
the discrete spectrum below $-\alpha^2/4$ is infinite, accumulates at $-\alpha^2/4$, and the eigenvalues 
$\lambda_k<-\alpha^2/4$  (enumerated in non-decreasing order with multiplicities taken into account) satisfy 
the estimate 
\begin{equation*}
\lambda_k \le -\frac{\alpha^2}{4} - 
\frac{\gamma(\theta)}{n_k^4},\qquad k\in\dN,
\end{equation*}
holds, where $\gamma(\theta) > 0$, $n_{k+1} := n_k^2 +n_k$ 
for $k\in\dN$, and $n_1=N$ with $N\in\dN$ sufficiently large.
\end{thm}

\subsection*{Acknowledgements}

The authors gratefully acknowledge financial support by the Austrian 
Science Fund (FWF), project
P 25162-N26, Czech Science Foundation (GA\v{C}R), project 
 14-06818S
and the Austria-Czech Republic cooperation grant  CZ01/2013.


\begin{thebibliography}{99}

\bibitem[BEL13]{BEL13}
J.~Behrndt, P.~Exner, and V.~Lotoreichik, Schr\"odinger operators with $\delta$- 
and $\delta'$-interactions on Lipschitz surfaces and chromatic numbers of associated partitions,
\emph{arXiv:1307.0074}.

\bibitem[BLL13]{BLL13} J.~Behrndt, M.~Langer, and V.~Lotoreichik,
Schr\"odinger operators
with $\delta$ and $\delta'$-potentials supported on hypersurfaces, 
\textit{Ann. Henri Poincar\'{e}}, \textbf{14} (2013), 385--423.


\bibitem[BS87]{BS87}
M.\,Sh.~Birman and M.\,Z.~Solomjak,
\textit{Spectral Theory of self-adjoint Operators in Hilbert Spaces},
Dordrecht, Holland, 1987.


\bibitem[BEKS94]{BEKS94}
J.\,F.~Brasche, P.~Exner, Yu.\,A.~Kuperin and P.~Seba,
Schr\"odinger operators with singular interactions,
\textit{J.\ Math.\ Anal.\ Appl.} \textbf{184} (1994), 112--139.

\bibitem[BM97]{BM97}
H.~Brezis and  M.~Marcus, 
Hardy's inequalities revisited. Dedicated to Ennio De Giorgi
\textit{Ann.\ Scuola\ Norm.\ Sup.\ Pisa\ Cl.\ Sci.} \textbf{25} (1997),
217--237.

\bibitem[BEW09]{BEW09} B.\,M.~Brown, M.\,.S.\,P.~Eastham and I.\,G.~Wood,
Estimates for the lowest eigenvalue of a star graph,
\textit{J.\ Math.\ Anal.\ Appl.} \textbf{354} (2009), 24--30.

\bibitem[CEK04]{CEK04}
G.~Carron, P.~Exner, and D.~Krej\v{c}i\v{r}ik, Topologically non-trivial quantum layers,
\textit{J.\ Math.\ Phys.} \textbf{45} (2004), 774--784.
%
\bibitem[D95]{D95}
E.\,B.~Davies,  \emph{Spectral theory and differential operators}, Cambridge University Press, 1995.

\bibitem[DR13]{DR13}
V.~Duchene and N.~Raymond, 
Spectral asymptotics of a broken delta interaction, \textit{arXiv:1312.5947}.


\bibitem[DEK01]{DEK01}
P.~Duclos, P.~Exner, and D.~Krej\v{c}i\v{r}ik, 
Bound states in curved quantum layers, \textit{Comm.\ Math.\ Phys.}
\textbf{223} (2001), 13--28.


%
\bibitem[EK02]{EK02}
P.~Exner and S.~Kondej, Curvature-induced bound states for a
$\delta$ interaction supported by a curve in $\mathbb{R}^3$, \textit{Ann.
H.~Poincar\'{e}} \textbf{3} (2002), 967--981.

\bibitem[EN03]{EN03}
P.~Exner and K.~N\v{e}mcov\'{a}, 
Leaky quantum graphs: approximations by point-interaction Hamiltonians,
\textit{J.\ Phys.\ A} \textbf{36} (2003), 10173--10193.


\bibitem[ET10]{ET10}
P.~Exner and M.~Tater,
Spectrum of Dirichlet Laplacian in a conical layer,
\emph{J.\ Phys.\ A} \textbf{43} (2010), 474023.

\bibitem[EY02]{EY02}
P.~Exner and K.~Yoshitomi,
Asymptotics of eigenvalues of the Schr\"odinger operator with a strong
$\delta$-interaction on a loop,
\textit{J.\ Geom.\ Phys.} \textbf{41} (2002), 344--358.

\bibitem[J13]{J13}
M.~Jing, A class of rotationally symmetric quantum layers of dimension 4, \textit{J.\ Math.\ Anal.\ Appl.} \textbf{397} (2013),  791--799.

\bibitem[K]{Kato}
T.~Kato, \emph{Perturbation theory for linear operators. Reprint of the 1980 edition},
Springer-Verlag, Berlin, 1995. 

\bibitem[KV08]{KV08}
H.~Kova\v{r}ik and S.~Vugalter, Estimates on trapped modes in deformed quantum layers,
\textit{J.\ Math.\ Anal.\ Appl.} \textbf{345} (2008),  566--572.

\bibitem[LL07]{LL07}
C.~Lin and Z.~Lu, Existence of bound states for layers built up over hypersurfaces in $\dR^{n+1}$,
\textit{J.\ Funct.\ Anal} \textbf{244} (2007), 1--25.

\bibitem[LR12]{LR12}
Z.~Lu and J.~Rowlett,  On the discrete spectrum of quantum layers, \textit{J.\ Math.\ Phys.} \textbf{53} (2012), 073519.

\bibitem[RS-IV]{RS-IV}
M.~Reed and B.~Simon, \emph{Methods of Modern Mathematical Physics, IV. Analysis of Operators}, Academic, New York, 1978.


\bibitem[S]{S}
K.~Schm\"udgen,
\emph{Unbounded Self-adjoint Operators on Hilbert Space},  
Springer, Dordrecht, 2012. 


\bibitem[S70]{S70}
B.~Simon, 
On the infinitude or finiteness of the number of bound states of an 
$N$-body quantum system, I, 
\textit{Helv.\ Phys.\ Acta} \textbf{43} (1970), 607--630.
\end{thebibliography}
\end{document}